\documentclass[12pt, reqno]{amsart}
\usepackage{amsmath, amsthm, amscd, amsfonts, amssymb, graphicx, color}
\usepackage[bookmarksnumbered, colorlinks, plainpages]{hyperref}
\hypersetup{colorlinks=true,linkcolor=red, anchorcolor=green, citecolor=cyan, urlcolor=red, filecolor=magenta, pdftoolbar=true}

\usepackage{amsfonts}
\usepackage{color,graphicx,shortvrb}
\usepackage{enumerate}
\usepackage{amsmath}
\usepackage{amssymb}
\usepackage{graphicx}

\usepackage{amsmath}

\newtheorem{theorem}{Theorem}[section]

\newtheorem{corollary}[theorem]{Corollary}

\newtheorem{proposition}[theorem]{Proposition}

\newtheorem{definition}[theorem]{Definition}
\newtheorem{remark}[theorem]{Remark}

\usepackage{centernot}

\usepackage{tikz}
\usetikzlibrary{arrows,matrix}

\newcommand{\uno}{\textbf{1}}

\def\RR{{\mathbb{R}}}
\def\CC{{\mathbb{C}}}
\def\NN{{\mathbb{N}}}

\def\J#1#2#3{ \left\{ #1,#2,#3 \right\} }

\def\min{{\rm Min}}

\newcommand{\xdownarrow}[1]{%
  {\left\downarrow\vbox to #1{}\right.\kern-\nulldelimiterspace}
}

\newcommand{\xDownarrow}[1]{%
  \ensuremath{%
    \left\Downarrow\vbox to #1{}\right.\kern-\nulldelimiterspace
  }%
}

\newcommand{\xUparrow}[1]{%
  \ensuremath{%
    \left\Uparrow\vbox to #1{}\right.\kern-\nulldelimiterspace
  }%
}

\makeatletter
\newcommand{\xRightarrow}[2][]{\ext@arrow 0359\Rightarrowfill@{#1}{#2}}
\makeatother

\makeatletter
\newcommand{\xLeftarrow}[2][]{\ext@arrow 0359\Leftarrowfill@{#1}{#2}}
\makeatother

\usepackage{enumerate}
\usepackage{amsmath}
\usepackage{amssymb}

\begin{document}
\title[Linear maps preserving extreme points]{Linear maps between C$^*$-algebras preserving extreme points and strongly linear preservers}

\author[Burgos]{Mar{\'i}a J. Burgos}

\address{Campus de Jerez, Facultad de Ciencias Sociales y de la Comunicaci\'on
Av. de la Universidad s/n, 11405 Jerez, C\'adiz, Spain}

\email{maria.burgos@uca.es}

\author[M{\'a}rquez-Garc{\'i}a]{Antonio C. M{\'a}rquez-Garc{\'i}a}

\address{Departamento de Matem\'aticas, Universidad de Almer\'ia, 04120 Almer\'ia, Spain}

\email{acmarquez@ual.es}

\author[Morales-Campoy]{Antonio Morales-Campoy}

\address{Departamento de Matem\'aticas, Universidad de Almer\'ia, 04120 Almer\'ia, Spain}
\email{amorales@ual.es}

\author[Peralta]{Antonio M. Peralta}
\address{Departamento de An{\'a}lisis Matem{\'a}tico, Universidad de Granada,\\
Facultad de Ciencias 18071, Granada, Spain}

\email{aperalta@ugr.es}

\thanks{Authors partially supported by the Spanish Ministry of Economy and Competitiveness project no. MTM2014-58984-P and Junta de Andaluc\'{\i}a grants FQM375, FQM194. The second author is supported by a Plan Propio de Investigaci{\' o}n grant from University of Almer{\' i}a. The fourth author acknowledges the support from the Deanship of Scientific Research at King Saud University (Saudi Arabia) research group no. RG-1435-020.}

\keywords{C$^*$-algebra, JB$^*$-triple, triple homomorphism, linear preservers, extreme points preserver, strongly Brown-Pedersen quasi-invertibility preserver}

\subjclass[2010]{47B49, 15A09, 46L05, 47B48}

\begin{abstract} We study new classes of linear preservers between C$^*$-algebras and JB$^*$-triples. Let $E$ and $F$ be JB$^*$-triples with $\partial_{e} (E_1)$. We prove that every linear map $T:E\to F$ strongly preserving Brown-Pedersen quasi-invertible elements is a triple homomorphism. Among the consequences, we establish that, given two unital C$^*$-algebras $A$ and $B,$ for each linear map $T$ strongly preserving Brown-Pedersen quasi-invertible elements, then there exists a Jordan $^*$-homomorphism $S: A\to B$ satisfying $T(x) = T(1) S(x)$, for every $x\in A$. We also study the connections between linear maps strongly preserving Brown-Pedersen quasi-invertibility and other clases of linear preservers between C$^*$-algebras like Bergmann-zero pairs preservers, Brown-Pedersen quasi-invertibility preservers and extreme points preservers.
\end{abstract}

\vspace*{-1.5cm}

\maketitle
\section{Introduction}

Let $X$ be a Banach space. In many favorable cases, the set $\partial_{e} (X_1)$, of all extreme points of the closed unit ball, $X_1$, of $X$, reveals many of the geometric properties of the whole Banach space $X$. There are spaces $X$ with $\partial_{e} (X_1)=\emptyset$, however, the Krein-Milman theorem guarantees that $\partial_{e} (X_1)$ is non-empty when $X$ is a dual space.\smallskip

Let $A$ be a C$^*$-algebra. It is known that $\partial_{e} (A_1)\neq\emptyset$ if and only if $A$ is unital (see \cite[Theorem I.10.2$(i)$]{Tak}). When $A$ is commutative, the unitary elements in $A$ are precisely the extreme points of the closed unit ball of $A$. The same statement remains true when $A$ is a finite von Neumann algebra (cf. \cite[Lemma 2]{MaMo98}). For a general unital C$^*$-algebra $A$, every unitary element in $A$ is an extreme point of the closed unit ball of $A$, however, the reciprocal statement is, in general, false (for example a non-surjective isometry in $B(H)$ is not a unitary element in this C$^*$-algebra). A theorem due to R.V. Kadison establishes that the extreme points of the closed unit ball of a C$^*$-algebra $A$ are precisely the maximal partial isometries of $A$, i.e., partial isometries $e\in A$ satisfying $(1-ee^*) A (1-e^*e) =0$ (cf. \cite[Theorem I.10.2]{Tak}).\smallskip

Let $A$ and $B$ be unital C$^*$-algebras. One of the consequences derived from the Russo-Dye theorem assures that a linear mapping $T: A \to B$ mapping unitary elements in $A$ to unitary elements in $B$ admits a factorization of the form $T(a) = u S (a)$ ($a\in A$), where $u$ is a unitary in $B$ and $S$ is a unital Jordan $^*$-homomorphism (cf. \cite[Corollary 2]{RuDye}). We recall that a linear map $T:A \to B$ between Banach algebras is a \emph{Jordan homomorphism} if $T(a^2)=T(a)^2$, for all $a\in A$ (equivalently, $T(ab+ba)=T(a)T(b)+T(b)T(a)$, for all $a,b\in A$). If $A$ and $B$ are unital, $T$ is called \emph{unital} if $T(1)=1$. If $A$ and $B$ are C$^*$-algebras and $T(a^*)=T(a)^*$, for every $a\in A$, then $T$ is called \emph{symmetric}. Symmetric Jordan homomorphisms are named \emph{Jordan $^*$-homomorphisms}.\smallskip

Consequently, the problem of studying the linear maps $T:A\to B$ such that $T(\partial_{e} (A_1))\subseteq \partial_{e} (B_1)$ is a more general challenge, which is directly motivated by the just mentioned consequence of the Russo-Dye theorem. We only know partial answers to this problem. Concretely, V. Mascioni and L. Moln\'{a}r studied the linear maps on a von Neumann factor $M$ which preserve the extreme points of the unit ball of $M$. They prove that if $M$ is infinite, then every linear mapping $T$ on $M$ preserving extreme points admits a factorization of the form $T(a) = u S (a)$ ($a\in M$), where $u$ is a (fixed) unitary in $M$ and $S$ either is a unital $^*$-homomorphism or a unital $^*$-anti-homomorphism (see \cite[Theorem 1]{MaMo98}). Theorem 2 in \cite{MaMo98} states that, for a finite von Neumann algebra $M$, a linear map $T:M\to M$ preserves extreme points of the unit ball of $M$ if and only if there exist a unitary operator $u\in M$ and a unital Jordan $^*$-homomorphism $S: M\to M$ such that $T(a) = u S (a)$ ($a\in A$). In \cite{LaMa}, L.E. Labuschagne and V. Mascioni study linear maps between C$^*$-algebras whose adjoints preserve extreme points of the dual ball.\smallskip

The above results of Mascioni and L. Moln\'{a}r are the most conclusive answers we know about linear maps between unital C$^*$-algebras preserving extreme points. In this note we shall revisit the problem in full generality. We present several counter-examples to illustrate that the conclusions proved by Mascioni and Moln\'{a}r for von Neumann factors need not be true for linear mappings preserving extreme points between unital C$^*$-algebras (compare Remarks \ref{remark v is not necessarily unitary} and \ref{remark extreme points preserver is not strongly BP-preserver}). It seems natural to ask whether a different class of linear preservers satisfies the same conclusions found by Mascioni and Moln\'{a}r.\smallskip

Every unital Jordan homomorphism between Banach algebras \emph{strongly preserves invertibility}, that is, $T(a^{-1})=T(a)^{-1}$, for every invertible element $a\in A$. Moreover, Hua's theorem (see \cite{Hua}) states that every unital additive map between fields that strongly preserves invertibility is either an isomorphism or an anti-isomorphism.\smallskip

Let $A$ be a Banach algebra. Recall that an element $a\in A$ is called \emph{regular} if there is $b$ in $A$ satisfying $aba=a$ and $b=bab$.
Given $a$ and $b$ in a C$^*$-algebra $A$, we shall say that $b$ is a \emph{Moore-Penrose inverse} of $a$ if $a=aba$, $b=bab$ and $ab$ and $ba$ are self-adjoint. It is known that every regular element $a$ in $A$ admits a unique Moore-Penrose inverse that will be denoted by $a^\dag$ (\cite{HarMb92}). Let $A^\dag$ be the set of regular elements in the C$^*$-algebra $A$.\smallskip

We say that a linear map $T$ between C$^*$-algebras $A$ and $B$ \emph{strongly preserves Moore-Penrose invertibility} if $T(a^\dag)=T(a)^\dag$, for all $a\in A^\dag$. It is known that every Jordan $^*$-homomorphism strongly preserves Moore-Penrose invertibility. In \cite{Mb09}, M. Mbekhta proved that a surjective unital bounded linear map from a real rank zero C$^*$-algebra to a prime C$^*$-algebra strongly preserves Moore-Penrose invertibility if and only if it is either an $^*$-homomorphism or an $^*$-anti-homomorphism. Recently in \cite{BMM11} the first three authors of this note show that a linear map $T$ strongly preserving Moore-Penrose invertibility between C$^*$-algebras $A$ and $B$, is a Jordan $^*$-homomorphism multiplied by a regular element of $B$ commuting with the image of $T$, whenever the domain C$^*$-algebra $A$ is unital and linearly spanned by its projections, or when $A$ is unital and has real rank zero and $T$ is bounded. It is also proved that every bijective linear map strongly preserving Moore-Penrose invertibility from a unital C$^*$-algebra with essential socle is a Jordan $^*$-isomorphism multiplied by an involutory element. The problem for linear maps strongly preserving Moore-Penrose invertibility between general C$^*$-algebras remains open.\smallskip

The set, $A_q^{-1},$ of quasi-invertible elements in a unital C$^*$-algebra $A$ was introduced by L. Brown and G.K. Pedersen as the set $A^{-1} \mathfrak{A_1} A^{-1}$, where $A^{-1}$ and $\mathfrak{A_1}$ denote the set of invertible elements in $A$ and the set of extreme points of the closed unit ball of $A$, respectively (see \cite{BroPe95}). It is known that $a\in A_q^{-1}$ if and only if there exists $b\in A$ such that $B(a,b) =0$ (cf. \cite[Theorem 1.1]{BroPe95} and \cite[Theorem 11]{JamSiddTah2013}), where $B(a,b)$ denotes the Bergmann operator on $A$ associated with $(a,b)$ (see section \ref{Sec: prelim} for details and definitions).\smallskip

The notion of quasi-invertible element was extended by F.B. Jamjoom, A.A. Siddiqui, and H.M. Tahlawi to the wider setting of JB$^*$-triples.  An element $x$ in a JB$^*$-triple $E$ is called \emph{Brown-Pedersen quasi-invertible} if there exists $y\in E$ such that $B(x,y)=0$ (cf. \cite{JamSiddTah2013}). The element $y$ is called a Brown-Pedersen quasi-inverse of $x$. It is known that $B(x,y)=0$ implies $B(y,x)=0$.  Moreover, the Brown-Pedersen quasi-inverse of an element is not unique. Indeed, if $B(x,y)=0$ then it can be checked that $B(x,Q(y)(x))=0$, so for any Brown-Pedersen quasi-inverse $y$ of $x$, $Q(y)(x)$ also is a Brown-Pedersen quasi-inverse of $x$. It is established in \cite[Theorems 6 and 11]{JamSiddTah2013} that an element $x$ in $E$ is Brown-Pedersen quasi-invertible if, and only if, it is (von Neumann) regular and its range tripotent is an extreme point of the closed unit ball of $E$, equivalently, there exists a complete tripotent $v\in E$ such that $x$ is positive and invertible in $E_2(v)$. Every regular element $x$ in $E$ admits a unique generalized inverse, which is denoted by $x^{\dag}$ (see sections \ref{Sec: prelim} for more details). In particular, the set, $E_q ^{-1}$, of all Brown-Pedersen quasi-invertible elements in $E$ contains all extreme points of the closed unit ball of $E$.\smallskip

We consider in this paper a new class of linear preserver. A linear map $T$ between JB$^*$-triples \emph{strongly preserves Brown-Pedersen quasi-invertibility} if $T$ preserves Brown-Pedersen quasi-invertibility and $T(x^{\wedge}) = T(x)^{\wedge}$ for every $x\in E^{-1}_q$. In the main result of this note we prove the following:  Let $A$ and $B$ be unital $C^*$-algebras. Let $T:A\to B$ be a linear map strongly preserving Brown-Pedersen quasi-invertible elements. Then there exists a Jordan $^*$-homomorphism $S: A\to B$ satisfying $T(x) = T(1) S(x)$, for every $x\in A$ (see Theorem \ref{thm them BP-primeCstar necessary condition unital Calgebras}).\smallskip

In section \ref{sec:5} we also explore the connections between linear maps strongly preserving Brown-Pedersen quasi-invertibility and other clases of linear preservers between C$^*$-algebras like Bergmann-zero pairs preservers, Brown-Pedersen quasi-invertibility preservers, and extreme points preservers.\smallskip

The reader should have realized at this point, that novelties here rely on results and tools of Jordan theory and JB$^*$-triples (see section \ref{Sec: prelim} for definitions). The research on linear preservers on C$^*$-algebras benefits from new results on linear preservers on JB$^*$-triples. In Theorem \ref{non-empty-extr} we prove that every linear map strongly preserving regularity between JB$^*$-triples $E$ and $F$ with $\partial_{e}(E_1)\neq\emptyset$, is a triple homomorphism (i.e. it preserves triple products). We complement this result by showing that the same conclusion remains true for every bounded linear operator strongly preserving regularity from a weakly compact JB$^*$-triple into another JB$^*$-triple (see Theorem \ref{weakly-compact}). The assumption of continuity cannot be dropped in the result for weakly compact JB$^*$-triples (cf. Remark \ref{cexample 42}). The most significan result (Theorem \ref{thm non-empty-extr strongly BP preserver triples}) assures that every linear map strongly preserving Brown-Pedersen quasi-invertible elements between JB$^*$-triples $E$ and $F$, with $\partial_{e}(E_1)\neq\emptyset$, is a triple homomorphism.

\section{Preliminaries}\label{Sec: prelim}

As we have commented in the introduction, in this paper we employ techniques and results in JB$^*$-triple theory to study new classes of linear preservers between C$^*$-algebras in connection with linear maps preserving extreme points. For this purpose, we shall regard every C$^*$-algebra as an element in the wider class of JB$^*$-triples. Following \cite{Ka}, a JB$^*$-triple is a complex Banach space $E$ together with a continuous triple product $\J ... : E\times E\times E \to E,$ which is conjugate linear in the middle variable
and symmetric and bilinear in the outer variables satisfying that,
\begin{enumerate}[{\rm (a)}] \item $L(a,b) L(x,y) = L(x,y) L(a,b)
+ L(L(a,b)x,y)
 - L(x,L(b,a)y),$
where $L(a,b)$ is the operator on $E$ given by $L(a,b) x = \J
abx;$ \item $L(a,a)$ is an hermitian operator with non-negative
spectrum; \item $\|L(a,a)\| = \|a\|^2$.\end{enumerate}

For each $x$ in a JB$^*$-triple $E$, $Q(x)$ will stand for the
conjugate linear operator on $E$ defined by the assignment
$y\mapsto Q(x) y = \J xyx$.\smallskip

The \emph{Bergmann operator}, $B(x,y)$, associated with a pair of elements $x,y\in E$ is the mapping defined by $$B(x,y)=I_E-2L(x,y)+Q(x)Q(y).$$

Every C$^*$-algebra is a JB$^*$-triple via the triple product given
by \begin{equation}\label{e-triprod-Cstar}
2 \J xyz = x y^* z +z y^* x,
\end{equation} and every JB$^*$-algebra is a
JB$^*$-triple under the triple product \begin{equation}\label{e
triple product JB-alg} \J xyz = (x\circ y^*) \circ z + (z\circ
y^*)\circ x - (x\circ z)\circ y^*.
\end{equation}

It is worth mentioning that, by the \emph{Kaup-Banach-Stone} theorem, a
linear surjection between JB$^{*}$-triples is an isometry if and only if it is a triple
isomorphism (compare \cite[Proposition 5.5]{Ka} or \cite[Corollary 3.4]{BePe} or \cite[Theorem 2.2]{FerMarPe}). We recall that a linear map $T:E\to F$ between JB$^*$-triples is a \emph{triple homomorphism} if $$T(\{x,y,z\})=\{T(x),T(y),T(z)\}\quad \mbox{for every } x,y,z\in E.$$
It follows, among many other consequences, that when a
JB$^{*}$-algebra $J$ is a JB$^{*}$-triple for a suitable triple
product, then the latter coincides with the one defined in
\eqref{e triple product JB-alg}.\smallskip

A JBW$^*$-triple is a JB$^*$-triple which is also a dual Banach
space (with a unique isometric predual \cite{BarTi}). It is known
that the triple product of a JBW$^*$-triple is separately weak$^*$
continuous \cite{BarTi}. The second dual of a JB$^*$-triple $E$ is
a JBW$^*$-triple with a product extending the product of $E$
\cite{Di}.\smallskip

An element $e$ in a JB$^*$-triple $E$ is said to be a
\emph{tripotent} if $\J eee =e$. Each tripotent $e$ in $E$ gives
raise to the following decomposition of $E$,
$$E= E_{2} (e) \oplus E_{1} (e) \oplus E_0 (e),$$ where for
$i=0,1,2,$ $E_i (e)$ is the $\frac{i}{2}$ eigenspace of $L(e,e)$
(compare \cite[Theorem 25]{loos}). The natural projections of $E$
onto $E_i(e)$ will be denoted by $P_i(e)$. This decomposition is
termed the \emph{Peirce decomposition} of $E$ with respect to the
tripotent $e.$ The Peirce decomposition satisfies certain rules
known as \emph{Peirce arithmetic}: $$\J {E_{i}(e)}{E_{j}
(e)}{E_{k} (e)}\subseteq E_{i-j+k} (e),$$ if $i-j+k \in \{
0,1,2\}$ and is zero otherwise. In addition, $$\J {E_{2}
(e)}{E_{0}(e)}{E} = \J {E_{0} (e)}{E_{2}(e)}{E} =0.$$
We observe that, for a tripotent $e\in E$, $B(e,e)=P_0(e)$.\smallskip

The Peirce space $E_2 (e)$ is a JB$^*$-algebra with product
$x\circ_e y := \J xey$ and involution $x^{\sharp_e} := \J exe$.
\smallskip

 A tripotent $e$ in $E$ is called \emph{complete} if the equality $E_0(E)=0$
 holds. When $E_2(e)=\CC e \neq \{0\},$ we say that $e$ is \emph{minimal}.\smallskip

For each element $x$ in a JB$^*$-triple $E$, we shall denote
$x^{[1]} := x$, $x^{[3]} := \J xxx$, and $x^{[2n+1]} := \J
xx{x^{[2n-1]}},$ $(n\in \NN)$. The symbol $E_x$\label{equ subtriple single generated} will stand for the
JB$^*$-subtriple generated by the element $x$. It is known that
$E_x$ is JB$^*$-triple isomorphic (and hence isometric) to $C_0
(\Omega)$ for some locally compact Hausdorff space $\Omega$
contained in $(0,\|x\|],$ such that $\Omega\cup \{0\}$ is compact,
where $C_0 (\Omega)$ denotes the Banach space of all
complex-valued continuous functions vanishing at $0.$ It is also
known that if $\Psi$ denotes the triple isomorphism from $E_x$
onto $C_{0}(\Omega),$ then $\Psi (x) (t) = t$ $(t\in \Omega)$ (cf.
\cite[Corollary 4.8]{Ka0}, \cite[Corollary 1.15]{Ka} and
\cite{FriRu}). The set $\overline{\Omega }=\hbox{Sp} (x)$ is
called the \emph{triple spectrum} of $x$. We should note that $C_0
(\hbox{Sp} (x)) = C(\hbox{Sp}(x))$, whenever $0\notin \hbox{Sp}
(x)$.\smallskip

Therefore, for each $x\in E$, there exists a unique element $y\in
E_x$ satisfying that $\J yyy =x$. The element $y,$ denoted by
$x^{[\frac13 ]}$, is termed the \emph{cubic root} of $x$. We can
inductively define, $x^{[\frac{1}{3^n}]} =
\left(x^{[\frac{1}{3^{n-1}}]}\right)^{[\frac 13]}$, $n\in \NN$.
The sequence $(x^{[\frac{1}{3^n}]})$ converges in the weak$^*$
topology of $E^{**}$ to a tripotent denoted by $r(x)$ and called
the \emph{range tripotent} of $x$. The tripotent $r(x)$ is the
smallest tripotent $e\in E^{**}$ satisfying that $x$ is positive
in the JBW$^*$-algebra $E^{**}_{2} (e)$ (compare \cite[Lemma
3.3]{EdRu}).\smallskip

Regular elements in Jordan triple systems and JB$^*$-triples have been deeply studied in \cite{FerGarSanMo92, Loos, Ka96} and \cite{BurKaMoPeRa08}.
An element $a$ in a JB$^*$-triple $E$ is called \emph{von Neumann
regular} if there exists (a unique) $b\in E$ such that $Q(a)(b)=a$,
$Q(b)(a)=b$ and $Q(a)Q(b)=Q(b)Q(a)$, or equivalently $Q(a)(b)=a$ and
$Q(a)(b^{[3]})=b$. The element $b$ is called the \emph{generalized
inverse} of $a$. We observe that every tripotent $e$ in $E$ is von
Neumann regular and its generalized inverse coincides with it.\smallskip

Throughout this note, we shall denote by $E^\wedge$ the set of regular elements in a JB$^*$-triple $E$, and for an element $a\in E^\wedge$, $a^\wedge$ will stand for its generalized inverse.\smallskip

To simplify notation, for a C$^*$-algebra $A$, let $E_A$ denote the JB$^*$-triple with underlaying Banach space $A$, and triple product defined by (\ref{e-triprod-Cstar}).
Let $a$ be an element in $E_A$. Then the mapping $Q(a)$ is given by $Q(a)(x) = \J axa = ax^*a$. Thus, $a$ is Moore-Penrose invertible in $A$ with Moore-Penrose inverse $b$ if, and only if, $a\in E_{A}^{\wedge}$ and $a^\wedge=(a^{\dag})^*=(a^*)^{\dag}$.\smallskip

Every triple homomorphism   $T:E\to F$  between JB$^*$-triples \emph{strongly preserves regularity}, that is, $T(x^\wedge)=T(x)^\wedge$ for every $ x \in E^\wedge$.
In \cite{BurMarMor12}, the authors characterized the triple homomorphism between $C^*$-algebras as the linear maps strongly preserving regularity. As a consequence, it is proved that a self-adjoint linear map from a unital C$^*$-algebra $A$ into a C$^*$-algebra $B$ is a triple homomorphism if and only if it strongly preserves Moore-Penrose invertibility (\cite[Theorem 3.5]{BurMarMor12}).\medskip


\section{Linear maps strongly preserving regularity on  JB$^*$-triples}

It is known that a non-zero element $a$ in a JB$^*$-triple $E$, is von Neumann regular if, and only if, $Q(a)(E)$ is closed, if and only if, the range tripotent $r(a)$ of $a$ lies in $E$ and $a$ is positive and invertible element in the JB$^*$-algebra
$E_2(e)$ (cf. \cite{FerGarSanMo92}, \cite{Ka96} or \cite{BurKaMoPeRa08}). Moreover, when $a$ is von Neumann regular,  $$L(a, a^{\wedge}) = L(a^{\wedge},a) = L(r(a), r(a)),$$ and $$Q(a)Q(a^{\wedge})=Q(a^{\wedge})Q(a)=P_2(r(a))$$ (see \cite[Lemma 3.2]{Ka01} or \cite[Theorem 3.4]{BurKaMoPeRa08}).
Recall that an element $a$ in a unital Jordan algebra $J=(J,\circ)$
is \emph{invertible} if there exists a (unique) element $b\in J$ such that $a\circ b=\uno$ and $a^2\circ
b=a$, equivalently $U_a$  is invertible with inverse $U_{b}$, where
$U_a$ is defined by $U_a(x)=2a\circ(a\circ x)-a^2\circ x$ (\cite[Theorem 13]{Jac}). If $a$ is invertible, its inverse is denoted by $a^{-1}$.
Moreover if $a$ and $b$ are invertible elements in the Jordan
algebra $J$ such that $a-b^{-1}$ is also invertible, then
$a^{-1}+(b^{-1}-a)^{-1}$ is invertible, and the Hua's identity
\begin{equation}\label{Hua}
\left(a^{-1}-(a-b^{-1})^{-1}\right) ^{-1}= \left(a^{-1}+(b^{-1}-a)^{-1}\right) ^{-1}=a-U_a(b)
\end{equation} holds
(see \cite[page 54, Exercise 3]{Jac}).\smallskip

A linear map $T: E\to F$ between JB$^*$-triples \emph{strongly preserves regularity} if $T(x^{\wedge}) = T(x)^{\wedge}$ for every $x\in E^{\wedge}$.\smallskip

The next result is inspired in \cite[Lemma 3.1]{BMM11}.

\begin{proposition}\label{cubes}
Let $E$ and $F$ be JB$^*$-triples, and let $T :E\to F$ be a linear map such that  $T(x^\wedge)=T(x)^\wedge$ for every $ x \in E^\wedge$. Then $$T(x^{[3]})=T(x)^{[3]},$$for every $ x \in E^\wedge$.
\end{proposition}

\begin{proof}
Let $x\in E^\wedge \setminus \{0\}$. Let $e=r(x)$ the range tripotent of $x$. As we have just mentioned, $x$ is positive and
invertible in the JB$^*$-algebra $E_2(e)$, with inverse $x^\wedge$, and $0\notin Sp(x)$. We identify $E_x$ (the JB$^*$-subtriple of $E$ generated by $x$) with $C(\hbox{Sp} (x))$ in such a way that $x$ corresponds to the function $t\mapsto t$. Hence for every $\lambda \in \CC$ with
$0<|\lambda|<||x^\wedge||^{-2}$, the element $\lambda x^\wedge-x$ is invertible in $E_x$, and hence invertible in $E_2(e)$, with inverse $(\lambda x^\wedge -x)^\wedge$. In this case, $x^\wedge+ (\lambda x^\wedge -x)^\wedge$ is invertible in $E_x$ (and in $E_2(e)$).\smallskip

Further, the inverses of $x-\lambda x^\wedge$ and $x^\wedge-(x-\lambda x^\wedge)^\wedge$ in $E_x$ (or in $E_2(e)$) are their generalized inverses in $E$ (let us recall that the triple product induced on  $E_2(e)$ by the Jordan $^*$-algebra structure coincides
with its original triple product, and $Q(x)=U_x \circ \sharp$, for every $x\in E$). By Hua's identity (cf. \eqref{Hua}), applied to $a= x$ and $b= \lambda^{-1} x$,
we obtain $$x-\lambda^{-1} x^{[3]}=\left(x^\wedge-(x-\lambda x^\wedge)^\wedge\right)^\wedge.$$

Let $x\in E^\wedge$. We may assume that $T(x)\neq 0$. Since $T$ strongly preserves
regularity, $T(x)^\wedge=T(x^\wedge)$. Thus, for
$\lambda\in\CC$ with $0<
|\lambda|<\min\{||x^{\wedge}||^{-2},||T(x)^{\wedge}||^{-2}\}$, we have
$$T(x)-\lambda ^{-1} T(x)^{[3]}=\left(T(x)^{\wedge}-(T(x)-\lambda
T(x)^{\wedge})^{\wedge}\right)^{\wedge}.$$ Since $T$ is linear and strongly preserves regularity, it follows that
$$T(x)-\lambda ^{-1} T(x)^{[3]}=(T(x)^\wedge-(T(x)-\lambda
T(x)^\wedge)^\wedge)^\wedge=(T(x^\wedge)-T(x-\lambda
x^\wedge)^\wedge)^\wedge $$ $$= T(\left(x^\wedge-(x-\lambda x^\wedge)^\wedge\right)^\wedge) =T(x)-\lambda^{-1}T(x^{[3]}),$$ and thus $T(x^{[3]})=T(x)^{[3]}$.
\end{proof}

Recall that two elements $a, b$ in a JB$^*$-triple $E$ are \emph{orthogonal} (written as $a\perp b$) if $L(a, b) = 0$ (see \cite[Lemma 1]{BurFerGarMarPe} for several equivalent reformulations).

\begin{remark}
	Let $T:E\to F$ be a linear map between JB$^*$-triples. Assume that $T$ strongly preserves regularity. Then $T$ preserves the orthogonality relation on regular elements. Indeed, given $a,b\in E^\wedge$, such that $a\perp b$, it can be easily seen that $$(a+\alpha b)^{\wedge}=a^\wedge + \alpha ^{-1} b^\wedge ,$$ for every $\alpha\in \RR \setminus \{0\}.$
	By assumption $T(a^\wedge + \alpha ^{-1} b^\wedge)=T(a+\alpha b)^\wedge$. In particular
	$$\{T(a)+\alpha T(b),T(a)^\wedge + \alpha ^{-1} T(b)^\wedge, T(a)+\alpha T(b)\}=T(a)+\alpha T(b).$$
It follows from the above identity that
	\begin{eqnarray*}
		& & 2\alpha\{T(a),T(a)^\wedge,T(b)\}+2\{T(a),T(b)^\wedge,T(b)\}\\ & &- \alpha^{-1}\{T(a),T(b)^\wedge,T(a)\}+\alpha^2\{T(b),T(a)^\wedge,T(b)\}=0,
	\end{eqnarray*} for every $\alpha\in \RR \setminus \{0\}.$ Therefore
	$$\{T(a),T(a)^\wedge,T(b)\}=0, \qquad \{T(a),T(b)^\wedge,T(b)\}=0.$$
	Since $L(T(a),T(a)^\wedge)=L(r(T(a)))$, and  $L(T(b),T(b)^\wedge)=L(r(T(b)))$, it follows that $T(a)\perp T(b)$.
\end{remark}

Notice that a JB$^*$-triple might contains no non-trivial tripotents (consider, for example, the $C^*$-algebra $C_0(0,1]$ of all complex-valued continuous functions on $[0,1]$ vanishing at $0$). However, since the complete tripotents of a JB$^*$-triple $E$ coincide with the extreme points of its closed unit ball (see \cite[Lemma 4.1]{BraKaUp78} and \cite[Proposition 3.5]{KaUp77}, or \cite[Theorem 3.2.3]{Chu}), every JBW$^*$-triple contains a large set of complete tripotents. Actually, the set of all tripotents in a JBW$^*$-triple is norm-total (cf. \cite[Lemma 3.11]{Horn87}).\smallskip

Let us recall that, by Lemma 2.1 in \cite{JaSiTaPe15}, an element $a$ in a JB$^*$-triple $E$ is Brown-Pedersen quasi-invertible if and only if $a$ is regular and $\{a\}^{\perp}=\{0\}$, where $\{a\}^{\perp}=\{b\in E\colon a\perp b=0\}$.\smallskip


\begin{theorem}\label{non-empty-extr} Let $E$ and $F$ be JB$^*$-triples with $\partial_{e}(E_1)\neq\emptyset$.
Let $T:E\to F$ be a linear map strongly preserving regularity. Then $T$ is a triple homomorphism.
\end{theorem}

\begin{proof}
Pick a complete tripotent $e\in E$. For every $x\in E$, let  $\lambda\in \CC$, with $|\lambda|>||P_2(e)(x)||$. It is clear that $P_2(e)(x-\lambda e)=P_2(e)(x)-\lambda e$ is invertible in the unital JB$^*$-algebra $E_2(e)$. It follows from \cite[Lemma 2.2]{JaSiTaPe15} that $x-\lambda e$ is Brown-Pedersen quasi-invertible.  We know, by Proposition \ref{cubes}, that
$$ T\left((x-\lambda e)^{[3]}\right)=T(x-\lambda e)^{[3]}.$$
Since the above identity holds for every $\lambda\in \CC$, with $|\lambda|>||P_2(e)(x)||$, 
we deduce that $$T(x^{[3]})=T(x)^{[3]},$$ for every $ x \in E$. The polarization formula
\begin{equation}\label{eq polarization fla} 8\{x,y,z\}=\sum_{k=0}^{3}\sum_{j=1}^{2}i^k (-1)^j\left( x+i^k y+(-1)^j z \right)^{[3]},
\end{equation} and the linearity of $T$ assure that $T$ is a triple homomorphism.
\end{proof}

The particularization of the previous result to the setting of C$^*$-algebras seems to be a new result.

\begin{corollary}\label{cor non-empty-extr} Let $T: A\to B$ be a linear map strongly preserving regularity between C$^*$-algebras. Suppose that $\partial_{e}(A_1)\neq\emptyset$. Then $T$ is a triple homomorphism.$\hfill\Box$
\end{corollary}

\section{Maps strongly preserving regularity on weakly compact JB$^*$-triples}
The notions of compact and weakly compact elements in JB$^*$-triples is due to L. Bunce and
Ch.-H. Chu \cite{BuChu}. Recall that an element a in a JB$^*$-triple $E$ is said to be \emph{compact} or \emph{weakly compact} if the mapping $Q(a)$ is compact or weakly compact, respectively. These notions extend, in a natural way, the corresponding definitions in the settings of $C^*$- and JB$^*$-algebras.
 A JB$^*$-triple $E$ is weakly compact
(respectively, compact) if every element in $E$ is weakly compact
(respectively, compact).\smallskip

In a JB$^*$-triple, the set of weakly compact elements is, in general, strictly bigger than the set of compact elements (cf. \cite[Theorem 3.6]{BuChu}). A non-zero tripotent $e$ in $E$ is called minimal whenever $E_2(e) = \CC e$. The socle, $\rm{soc}(E)$, of a
JB$^*$-triple $E$ is the linear span of all minimal tripotents in $E$. Following \cite{BuChu}, the symbol $K_0(E)$ denotes the norm-closure of $\rm{soc}(E)$. By \cite[Lemma 3.3 and Proposition 4.7]{BuChu}, the triple ideal $K_0(E)$ coincides with the set of all weakly compact elements in $E$. Hence a JB$^*$-triple $E$ is weakly compact whenever $E=K_0(E)$.
Every finite sum of mutually orthogonal minimal tripotents in a JB$^*$-triple $E$ lies in the socle of $E$. It is also known that an element $a$ in a JB$^*$-triple $E$ is weakly compact if and only if $L(a, a)$ is a weakly compact operator (see \cite{BuChu}). Therefore, for each tripotent $e$ in the socle of $E$, $P_1(e)=2L(e, e) - P_2(e)= 2 L(e, e) - Q(e)^2$ is a weakly compact operator on $E$ (cf. \cite[\S 2]{FerMarPe12}).\smallskip

It is well known that every element in the socle of a JB$^*$-triple is regular. Moreover, for every JB$^*$-triple $E$,
$$E^\wedge+\rm{soc}(E)\subseteq E^\wedge.$$
Indeed, given $a\in E^\dag$ and $x\in \rm{soc}(E)$,
$$(a+x)-Q(a+x)(a^\wedge)=x-2\{ a,a^\wedge,x \}-\{x,a^\wedge,x\}\in \rm{soc}(E)\subseteq E^\dag.$$
By Mc Coy's Lemma (see \cite{M72}), $a+x\in E^\wedge.$
Let $E$, $F$ be JB$^*$-triples. Let us assume that $E$ has non-zero socle, and  let $T:E\to F$ be a linear map strongly preserving regularity. The polarization formula \eqref{eq polarization fla} and Proposition \ref{cubes} show that $T(\{x,y,z\})=\{T(x),T(y),T(z)\}$, whenever one of the elements $x$, $y$, or $z$ is regular and the others lie in the socle.

\begin{theorem}\label{weakly-compact} Let $E$, $F$ be JB$^*$-triples, with $E$ weakly compact. Let $T:E\to F$ be a bounded linear map strongly preserving regularity. Then $T$ is a triple homomorphism.
\end{theorem}

\begin{proof}
We know, from Proposition \ref{cubes}, that $T$ preserves cubes of regular elements. Since every element in the socle of a JB$^*$-triple is regular, is follows that $T(x^{[3]})=T(x)^{[3]},$ for every $x \in \rm{soc} (E)$. Since $E=K_0(E)=\overline{\rm{soc}(E)}$, the continuity of $T$, together with the norm continuity of the triple product prove that $T$ is a triple homomorphism.
\end{proof}

In the next example we show that the continuity assumption cannot be dropped from the hypothesis in the previous theorem (even in the setting of $C^*$-algebras).

\begin{remark}\label{cexample 42} Let $c_0$ denote the $C^*$-algebra of all scalar null sequences. It is clear that $c_0$ is a weakly compact JB$^*$-triple, with $\rm{soc}(c_0)=c_{00},$ i.e. the subspace of eventually zero sequences in $c_0$. Let $\{e_n\}$ denote the standard coordinate (Schauder) basis of $c_0$. We extend this basis, via Zorn's lemma, to an algebraic (Hamel) basis of $c_0$, say $B=\{e_n\}\cup\{z_n\}$.\smallskip

We define $T:c_0 \to c_0$ as the linear (unbounded) mapping given by
$$T(e_n)=e_n, \qquad T(z_n)=n z_n.$$ Clearly $T$ is not a triple homomorphism but it strongly preserves regularity. Let us notice that $c_0^\wedge=c_{00}$ and $T(c_{00}) = c_{00}$.
\end{remark}

\section{Linear maps strongly preserving Brown-Pedersen quasi-invertibility}\label{sec:5}

In \cite{FerMarPe12}, the authors proved that Bergmann operators can be used to characterize the relation of being orthogonal in JB$^*$-triples.
More concretely, it is proved in \cite[Proposition 7]{FerMarPe12} that, for any element $x$ in a JB$^*$-triple $E$ with $||x|| < \sqrt{2}$, the orthogonal annihilator of $x$ in $E$ coincides with the set of all fixed points of the Bergmann operator $B(x, x)$. It is also obtained, in the just quoted paper, that a norm one element $e$ in a JB$^*$-triple $E$ is a tripotent if, and only if, $B(e, e) (E) = \{e\}^\perp$ (cf. \cite[Proposition 9]{FerMarPe12}).\smallskip

Having in mind all the characterizations of tripotents and Brown-Pedersen quasi-invertible elements commented above, and recalling that extreme points of the closed unit ball of a JB$^*$-triple $E$ are precisely the complete tripotents in $E$, it can be deduced that the equivalence \begin{equation}\label{eq characterization of extreme points as Bergmann zero} e\in \partial_{e} (E_1) \Leftrightarrow B(e,e)=0,
\end{equation} holds for every $e\in E_1.$\smallskip

Let $T:E\to F$ be a linear map between JB$^*$-triples. We introduce the following definitions:

\begin{definition}\label{def BP preserving } $T$ \emph{preserves Brown-Pedersen quasi-invertibility} if $T(E_q^{-1}) \subseteq F_q^{-1}$, that is, $T$ maps Brown-Pedersen quasi-invertible elements in $E$ to Brown-Pedersen quasi-invertible elements in $F$.\smallskip
\end{definition}

\begin{definition}\label{def Bergmann-zero preserving} $T$ \emph{preserves Bergmann-zero pairs} if $$B(x,y)=0\Rightarrow B(T(x),T(y))=0.$$
\end{definition}

\begin{definition}\label{def strongly BP preserving} $T$ \emph{strongly preserves Brown-Pedersen quasi-invertibility} if $T$ preserves Brown-Pedersen quasi-invertibility and $T(x^{\wedge}) = T(x)^{\wedge}$ for every $x\in E^{-1}_q$.
\end{definition}

\begin{definition}\label{def extreme points preserving} $T$ \emph{preserves extreme points} if $T(\partial_{e}(E_1))\subseteq \partial_{e}(F_1)$.
\end{definition}

It is worth to notice that all definitions above make sense for linear operators between C$^*$-algebras. In this paper we employ Jordan techniques to study these kind of mappings and so, we set the above definitions in the most general setting.\smallskip

Suppose $T: E\to F$ is a linear mapping strongly preserving Brown-Pedersen quasi-invertibility between two JB$^*$-triples. Suppose $u\in \partial_{e}(E_1)$. Then $u$ is Brown-Pedersen quasi-invertible with $u^{\wedge} =u$. It follows from our assumptions that $T(u)$ is Brown-Pedersen quasi-invertible and $T(u)^{\wedge} = T(u^{\wedge}) = T(u)$. In such a case, $\{T(u),T(u),T(u)\}= Q(T(u)) (T(u))= T(u)$ is a tripotent and Brown-Pedersen quasi-invertible, which implies that $T(u)\in \partial_{e}(E_1)$ (cf. \cite[Lemma 2.1]{JaSiTaPe15}). We have therefore shown that every linear mapping between JB$^*$-triples strongly preserving Brown-Pedersen quasi-invertibility also preserves extreme points points.\smallskip

The characterization of the extreme points of the closed unit ball of a JB$^*$-triple given in \eqref{eq characterization of extreme points as Bergmann zero} implies that every linear mapping between JB$^*$-triples preserving Bergmann-zero pairs also preserves extreme points.\smallskip

Clearly, a linear mapping  $T: E\to F$ preserving Bergmann-zero pairs maps Brown-Pedersen quasi-invertible elements in $E$ to Brown-Pedersen quasi-invertible elements in $F$.\smallskip

Therefore, for every linear mapping $T$ between JB$^*$-triples the following implications hold:{
$$\begin{array}{ccc}
  \begin{array}{c}
                                                             \hbox{$T$ preserves} \\
                                                             \hbox{Bergmann-zero pairs}
                                                           \end{array} &  \xRightarrow{   \ \ \ \ \ \ \    } & \begin{array}{c}
                                                             \hbox{$T$ preserves} \\
                                                             \hbox{BP quasi-invertible elements}
                                                           \end{array}
   \\
  \begin{array}{ccc}
    \xDownarrow{0.55cm}\rule{0.4pt}{0cm} & 
    & 
  \end{array}
   &  & \begin{array}{ccc}
          ${{\tiny (Remark \ref{remark extreme points preserver is not strongly BP-preserver})}}$ & \centernot{\xDownarrow{0.55cm}\rule{0.4pt}{0cm}} & \xUparrow{0.55cm}\rule{0.4pt}{0cm}
        \end{array} \\
  \hbox{$T$ preserves extreme points} &  \begin{array}{c}
                                                     \xLeftarrow{   \ \ \ \ \ \ \    } \\
                                                     \centernot{\xRightarrow{\hbox{{\tiny (Remark \ref{remark extreme points preserver is not strongly BP-preserver})}}}} 
                                                   \end{array}
   & \begin{array}{c}
                                                       \hbox{$T$ strongly preserves} \\
                                                       \hbox{BP quasi-invertible elements}
                                                     \end{array}
\end{array}$$}

The other implications are, for the moment, unknown. We have already commented that V. Mascioni and L. Moln\'{a}r characterized the linear maps on a von Neumann factor $M$ preserving the extreme points of the unit ball of $M$ in \cite{MaMo98}. According to our terminology, they prove that, for a von Neumann factor $M$, a linear map $T:M\to M$ such that  $B(T(a),T(a))=0$ whenever $B(a,a)=0$, is a unital Jordan $^*$-homomorphism multiplied by a unitary element (see \cite[Theorem 1, Theorem 2]{MaMo98}).\smallskip

Suppose $T:E\to E$ is a linear mapping between JB$^*$-triples which preserves Bergmann-zero pairs. Given a Brown-Pedersen quasi-invertible element $x$, with generalized inverse $x^\wedge$,  we have   $$B(x,x^\wedge)=B(x^\wedge,x)=B(r(x),r(x))=0,$$ and hence $B(T(x),T(x^\wedge))=0.$ This shows that $$Q(T(x))(T(x^\wedge))=T(x),\quad \mbox{and}\quad Q(T(x^\wedge))(T(x))=T(x^\wedge).$$ However $T(x^\wedge)$ may not coincide, in general, with $T(x)^\wedge$. So, we cannot conclude that every linear Bergmann-zero pairs preserving is a strongly Brown-Pedersen quasi-invertibility preserver (cf. Remark \ref{remark extreme points preserver is not strongly BP-preserver}).\smallskip

We mainly focus our study on maps between C$^*$-algebras. Let $A$ be a unital $C^*$-algebra $A$. It is easy to see that, for an element $a$ in $A$
$$B(a,a)(x)=(1-aa^*)x(1-a^*a), \quad \mbox{for all }x\in A.$$
Moreover it is also a well known fact that the extreme points of the closed unit ball of $A$ are precisely those elements $v$ in $A$ for which
$(1-vv^*)A(1-v^*v)=\{0\}$ (see \cite[Theorem I.10.2]{Tak}).\smallskip

Let $T:A\to B$ be a linear map between unital $C^*$-algebras which preserves extreme points. Since for every unitary element $u\in A$, $B(u,u)=0$ it follows that $B(T(u),T(u))=0$, which, in particular, shows that $T(u)$ is a partial isometry. Hence, $T$ is automatically bounded and $\|T\|=1$ (cf. \cite[\S 3]{RuDye}). Therefore, for every self-adjoint element $a\in A$, we have $$\{T(e^{ita}),T(e^{ita}),T(e^{ita})\}=T(e^{ita}) \qquad (t\in \RR).$$
Differentiating both sides of the above identity with respect to $t$, we deduce that $$2 \{i T(a e^{ita}),T(e^{ita}),T(e^{ita})\}+ \{T(e^{ita}),i T(a e^{ita}),T(e^{ita})\} =i T(a e^{ita}),$$ and hence \begin{equation}\label{eq first derivative} 2 \{ T(a e^{ita}),T(e^{ita}),T(e^{ita})\}- \{T(e^{ita}), T(a e^{ita}),T(e^{ita})\} = T(a e^{ita}),
 \end{equation}for every $t\in \mathbb{R}$. For $t=0$, we get $$ 2 \{ T(a),T(1),T(1)\}- \{T(1), T(a),T(1)\} = T(a),$$ equivalently  \begin{equation}\label{eq new 4.6} T(a)=T(a)T(1)^*T(1)+T(1)T(1)^*T(a)-T(1)T(a)^*T(1),\end{equation} for every $a=a^*$ in $A$.\smallskip

Differentiating \eqref{eq first derivative} with respect to $t$, we obtain 
$$ T(a^2 e^{ita}) = 2 \{ T(a^2 e^{ita}),T(e^{ita}),T(e^{ita})\} - 4 \{ T(a e^{ita}), T(a e^{ita}),T(e^{ita})\} $$ $$ + 2 \{ T(a e^{ita}),T(e^{ita}), T(a e^{ita})\} + \{T(e^{ita}),  T(a^2 e^{ita}),T(e^{ita})\},$$ for every $t\in \mathbb{R}$. In the case $t=0$ we get $$ T(a^2) = 2 \{ T(a^2),T(1),T(1)\} - 4 \{ T(a), T(a),T(1)\} $$ $$ + 2 \{ T(a),T(1), T(a)\} + \{T(1),  T(a^2),T(1)\},$$ or equivalently, \begin{equation}\label{eq new 4.7} T(a^2) = T(a^2)T(1)^*T(1)+ T(1) T(1)^* T(a^2)  - 2 T(a)T(a)^*T(1) \end{equation} $$-2 T(1) T(a)^* T(a)  + 2 T(a)T(1)^*T(a)+ T(1)T(a^2)^*T(1),
$$ for every $a=a^*$ in $A$.\smallskip

Multiplying identity \eqref{eq new 4.6} by $T(1)^*$ from both sides, and taking into account that $T(1)$ is a (maximal) partial isometry, we deduce that
\begin{equation}\label{partial}T(1)^* T(a) T(1)^* =T(1)^*T(1)T(a)^*T(1)T(1)^*,\end{equation} for every self-adjoint element $a\in A$.

\begin{proposition}\label{p BP-primeCstar necessary condition v unitary} Let $A$ and $B$ be unital $C^*$-algebras. Let $T:A\to B$ be a linear map preserving extreme points. Suppose that $T(1)$ is a unitary in $B$. Then there exists a unital Jordan $^*$-homomorphism $S: A\to B$ satisfying $T(a) = T(1) S(a)$, for every $a\in A$.
\end{proposition}

\begin{proof}  By hypothesis $v= T(1)$ is a unitary in $B$. We deduce from \eqref{eq new 4.6} that $$ T(a)=v T(a)^*v,$$ for every self-adjoint element $a\in A$, and hence, by linearity,
\begin{equation}\label{eq 5.10} T(a)=v T(a^*)^*v, \hbox{ or equivalently, } v^* T(a)= T(a^*)^* v,
\end{equation} for every $a\in A$. Therefore, the mapping $S:A\to B$, given by $S(x):=v^*T(x)$, is symmetric ($S(x^*)= S(x)^*$), and $S(1) = v^* T(1) = v^* v =1$.\smallskip

Now, since $v^*v=1 = vv^*$, we deduce from \eqref{eq new 4.7} and \eqref{eq 5.10} that $$T(a^2)=  v T(a)^* T(a) ,$$
for every $a=a^*$ in $A$. Multiplying on the left by $v^*$ we obtain: $$S(a^2)= v^* v T(a)^* T(a) = T(a)^* T(a) = S(a)^* S(a) = S(a)^2,$$ for every $a=a^*$ in $A$, and hence $S$ is a Jordan $^*$-homomorphism. It is also clear that $T(a) = v v^* T(a) = v S(a),$ for every $a$ in $A$.
\end{proof}

We recall that, according to \cite[Proposition 1.6.3]{Sak}, for a C$^*$-algebra, $A$, the intersection $\partial_{e} (A_1)\cap A_{sa}$ is precisely the set of all self-adjoint unitary elements of $A$.

\begin{corollary}\label{c BP-primeCstar necessary condition symmetric} Let $A$ and $B$ be unital $C^*$-algebras. Let $T:A\to B$ be a symmetric linear map. If $T$ preserves extreme points then $T(1)$ is a self-adjoint unitary element in $B$ and there exists a unital Jordan $^*$-homomorphism $S: A\to B$ satisfying $T(a) = T(1) S(a)$, for every $a\in A$.
\end{corollary}

\begin{proof} Suppose that $T$ preserves extreme points. Since $T$ is symmetric, the element $T(1)$ must be a self-adjoint extreme point of the closed unit ball of $B$, and hence a self-adjoint unitary element. Proposition \ref{p BP-primeCstar necessary condition v unitary} assures that $S(a) := T(1) T(a)$ ($a\in A$) is a unital Jordan $^*$-homomorphism and $T (a) = T(1) S(a)$, for every $a\in A$.
\end{proof}

The next result gives sufficient conditions for the reciprocal statement of  Proposition \ref{p BP-primeCstar necessary condition v unitary} and Corollary \ref{c BP-primeCstar necessary condition symmetric}.

\begin{proposition}\label{p BP-primeCstar necessary condition symmetric} Let $T:A\to B$ be a linear map between unital $C^*$-algebras. Suppose that $T$ writes in the form $T = v S$, where $v$ is a unitary element in $B$ and $S: A\to B$ is a unital Jordan $^*$-homomorphism such that $B$ equals the C$^*$-algebra generated by $S(A)$. Then $T$ preserves extreme points.
\end{proposition}

\begin{proof} Suppose that $T = v S$, where $v$ is a unitary element in $B$ and $S: A\to B$ is a unital Jordan $^*$-homomorphism. Since $S^{**} : A^{**} \to B^{**}$ is a unital Jordan $^*$-homomorphism between von Neumann algebras (cf. \cite[Lemma 3.1]{Stor1965}), Theorem 3.3 in \cite{Stor1965} implies the existence of two orthogonal central projections $E$ and $F$ in $B^{**}$ such that $S_1 = S^{**}: A^{**}\to B^{**} E$ is a $^*$-homomorphism, $S_2 = S^{**}: A^{**}\to B^{**} F$ is a $^*$-anti-homomorphism, $E+F=1$ and $S^{**}=S_1+ S_2$. The equality $1= S(1) = S_1(1) + S_2 (1)$ implies that $S_1 (1) = E$ and $S_2 (1) =F$. \smallskip

Take $e\in \partial_{e}(A_1)$. We claim that $S(e) \in \partial_{e}(B_1)$. Indeed, the equalities $$(1-S(e)S(e)^*) S(A) (1-S(e)^*S(e)) $$ $$= (1-S_1(e)S_1(e)^*-S_2(e)S_2(e)^*) S(A) (1-S_1(e)^*S_1(e)-S_2(e)^*S_2(e))  $$ $$= (E-S_1(e)S_1(e)^*) S_1(A) (E-S_1(e)^*S_1(e)) $$ $$+ (F-S_2(e)S_2(e)^*) S_2(A) (F-S_2(e)^*S_2(e))  $$ $$= S_1 ((1-ee^*) A (1-e^*e))+S_2 ((1-e^*e) A (1-ee^*))= \{0\},$$ together with the fact that $B$ equals the C$^*$-algebra generated by $S(A)$ show that $S(e) \in \partial_{e}(B_1)$.\smallskip

Finally, given $e\in \partial_{e}(A_1)$ we know that $S(e) \in \partial_{e}(B_1)$, and hence $$ (1-T(e)T(e)^*) B (1-T(e)^* T(e)) = (1-vS(e)S(e)^*v) B (1-S(e^*) v^* v S(e))$$ $$= v (1-S(e)S(e)^*) v^* B (1-S(e^*) S(e)) $$ $$\subseteq v (1-S(e)S(e)^*) B (1-S(e^*) S(e)) =\{0\},$$ because $S(e) \in \partial_{e}(B_1)$. We have therefore shown that $T(e)\in \partial_{e}(B_1)$.
\end{proof}

Henceforth, $T:A\to B$ will denote a linear map between unital $C^*$-algebras which preserves extreme points, and we assume that $B$ is prime. Let $v= T(1)\in \partial_{e}(B_1).$ The assumption $B$ being prime implies that $vv^*=1$ or $v^*v=1$. We shall assume that $v^*v=1.$ When $vv^*=1$ we can apply Proposition \ref{p BP-primeCstar necessary condition v unitary}, otherwise (i.e. in the case $vv^*\neq 1$), we cannot deduce the same conclusions. Indeed, from \eqref{eq new 4.6} and \eqref{partial} we deduce that $vv^* T(a) = vT(a)^* v$, for every $a=a^*$ in $A$, and consequently $v^* T(a) = T(a)^* v$, for every $a\in A_{sa}$. Therefore, the operator $S= v^* T$ is unital and symmetric. Moreover, it follows from \eqref{eq new 4.7} that $$v v^* T(a^2)  = 2 T(a)T(a)^*v + 2 v T(a)^* T(a)  - 2 T(a)v^*T(a)- vT(a^2)^*v,$$
for every $a=a^*$ in $A$. Multiplying on the left by $v^*$, and having in mind that $S$ is symmetric, we obtain $$ S(a^2)=   2 S (a) S(a^*)+ 2 T(a)^* T(a)  - 2 S(a) S(a) - S((a^2)^*) = 2 T(a)^* T(a) - S(a^2),$$ which proves that $S(a^2) = T(a)^* T(a)\geq S(a)^2,$ for every $a=a^*$ in $A$.\smallskip

When $M$ is an infinite von Neumann factor, a linear map $T: M\to M$ preserves extreme points if and only if there exist a unitary $u$ in $M$ and a linear map $\Phi: M\to M$ which is either a unital $^*$-homomorphism or a unital $^*$-anti-homomorphism such that
$T(a) = u \Phi (a)$ ($a\in A$) \cite[Theorem 1]{MaMo98}. When $M$ is a finite von Neumann algebra, a linear map $T$ on $M$ preserves extreme points if and only if there exist a unitary $u$ in $M$ and a Jordan $^*$-homomorphism $\Phi: M\to M$ satisfying $T(a) = u \Phi (a)$ ($a\in A$) \cite[Theorem 2]{MaMo98}. Motivated by these results it is natural to ask whether a similar conclusion remains true for operators preserving extreme points between unital C$^*$-algebras. The next simple examples show that the answer is, in general, negative.

\begin{remark}\label{remark v is not necessarily unitary}

Let $H$ be an infinite dimensional complex Hilbert space. Suppose $v$ is a maximal partial isometry in $B(H)$ which is not a unitary.  The operator $T: \mathbb{C} \to B(H),$ $\lambda \mapsto \lambda v$, preserves extreme points, but we cannot write $T$ in the form $T = u \Phi$, where $u$ is a unitary in $B(H)$ and $\Phi$ is a unital Jordan $^*$-homomorphism.
\end{remark}

\begin{remark}\label{remark extreme points preserver is not strongly BP-preserver}
Under the assumptions of Remark \ref{remark v is not necessarily unitary}, let $v,w\in \partial_{e} (B(H)_1)$ such that $v^*v=1= w^*w$ and $vv^*\perp ww^*$. Let $A= \mathbb{C}\oplus^{\infty}\mathbb{C}$.  We consider the following operator $$T: A \to B(H)$$ $$T(\lambda,\mu) = \frac{\lambda}{2} (v+w) + \frac{\mu}{2} (v-w).$$ Clearly $T(1,1) = v$. Furthermore, every extreme point of the closed unit ball of $A$ writes in the form $(\lambda_0,\mu_0)$ with $|\lambda_0|= |\mu_0|=1$. Therefore $T(\lambda_0,\mu_0) = \frac{\lambda_0}{2} (v+w) + \frac{\mu_0}{2} (v-w) = \frac{\lambda_0+\mu_0}{2} v+ \frac{\lambda_0-\mu_0}{2} w$ satisfies $$T(\lambda_0,\mu_0)^*T(\lambda_0,\mu_0)= \left(\frac{\lambda_0+\mu_0}{2} v+ \frac{\lambda_0-\mu_0}{2} w\right)^* \left(\frac{\lambda_0+\mu_0}{2} v+ \frac{\lambda_0-\mu_0}{2} w\right)$$ $$= \frac{|\lambda_0+\mu_0|^2}{4} v^*v+ \frac{|\lambda_0-\mu_0|^2}{4} w^*w  = \left(\frac{|\lambda_0+\mu_0|^2}{4} + \frac{|\lambda_0-\mu_0|^2}{4}  \right) 1 $$ $$= \frac{2 (|\lambda_0|^2+ |\mu_0|^2)}{4} 1 = 1,$$ which proves that $T(\lambda_0,\mu_0)\in \partial_{e} (B(H)_1)$, and hence $T$ preserves extreme points.\smallskip

The mapping $T$ satisfies a stronger property. Elements $a$ and $b$ in the C$^*$-algebra $A$ satisfy $B(a,b) = 0$ if and only if $a b^*=1$. We observe that $A_q^{-1} = A^{-1}$, and hence an element $(\lambda,\mu)\in A_q^{-1}$ if and only if $\lambda\mu \neq 0$. Let us pick $a= (\lambda_0,\mu_0)\in A_q^{-1}$ (with $\lambda_0\mu_0 \neq 0$). Clearly, $a^{\wedge} = \left(\overline{\lambda_0^{-1}},\overline{\mu_0^{-1}}\right)$. It is easy to check that $$T(a^{\wedge})^* T(a)=  \left( \frac{\overline{\lambda_0}^{-1} +\overline{\mu_0}^{-1}}{2} v + \frac{\overline{\lambda_0}^{-1} -\overline{\mu_0}^{-1}}{2} w \right)^*  \left( \frac{\lambda_0 +\mu_0}{2} v + \frac{\lambda_0-\mu_0}{2} w  \right) $$ $$=\left( \frac{{\lambda_0}^{-1} +{\mu_0}^{-1}}{2} v^* + \frac{{\lambda_0}^{-1} -{\mu_0}^{-1}}{2} w^* \right)  \left( \frac{\lambda_0 +\mu_0}{2} v + \frac{\lambda_0-\mu_0}{2} w  \right) $$ $$=\frac14 \frac{(\lambda_0 +\mu_0)^2-(\lambda_0 -\mu_0)^2}{\lambda_0 \mu_0} 1 = 1,$$ and hence $B(T(a),T(a^{\wedge}))=0,$ which shows that $T$ preserves Bergmann-zero pairs.\smallskip

It is easy to check that $T(1,-1)= w$, and hence $v^* T(1,-1)=v^*w =0$, and $vv^* T(1,-1) = 0$. For $S= v^* T$ we have $S(1,-1)^2 = 0$ but $S((1,-1)^2) = S(1,1) =v$, that is, $S$ is not a Jordan homomorphism. We can further check that $T$ is not a triple homomorphism, for example, $ (1,0)$ is a tripotent in $A$ but $\|T(1,0)\|= \frac{1}{\sqrt{2}}$, and hence $T(1,0)$ is not a tripotent in $B(H)$.\smallskip

Finally, for $a= (\lambda_0,\mu_0)\in A_q^{-1}$ (with $\lambda_0\mu_0 \neq 0$), $T(a^{\wedge}) = \frac{\overline{\lambda_0}^{-1}}{2} (v+w) + \frac{\overline{\mu_0}^{-1}}{2} (v-w)$ need not coincide with $T(a)^{\wedge} = \left( \frac{\lambda_0}{2} (v+w) + \frac{\mu_0}{2} (v-w) \right)^{\wedge}$. Indeed, $T(2,1)= \frac32 v + \frac12 w = \frac{\sqrt{10}}{2} r$, where $r= \frac{3}{\sqrt{10}} v + \frac{1}{\sqrt{10}} w$ is the range tripotent of $T(2,1)$, and thus $T(2,1)^{\wedge}= \frac{2}{\sqrt{10}} r = \frac35 v + \frac15 w$. Clearly, $T((2,1)^{\wedge}) = T(1/2,1) = \frac34 v - \frac14 w$.
\end{remark}

The counter-examples provided by Remark \ref{remark extreme points preserver is not strongly BP-preserver} point out that the conclusions found by Mascioni and Moln\'{a}r for linear maps preserving extreme points on infinite von Neumann factor (cf. \cite{MaMo98}) are not expectable for general C$^*$-algebras. We shall show that a more tractable description is possible for linear maps strongly preserving Brown-Pedersen quasi-invertibility. The proofs are based on the JB$^*$-triple structure underlying every C$^*$-algebra.\smallskip

The following variant of Proposition \ref{cubes} follows with similar arguments, its proof is outlined here.

\begin{proposition}\label{p cubes strongly BP preserver}
Let $E$ and $F$ be JB$^*$-triples, and let $T:E\to F$ be a non-zero linear map strongly preserving Brown-Pedersen quasi-invertible elements, that is, $T(x^\wedge)=T(x)^\wedge$ for every $ x \in E^{-1}_{q}$. Then $$T(x^{[3]})=T(x)^{[3]},$$ for every $ x \in E^{-1}_q$.
\end{proposition}

\begin{proof} Let $x$ be an element in $E^{-1}_q$, and let $e=r(x)\in \partial_{e} (E_1)$ denote its range tripotent. For each $0<\lambda <\|x^{\wedge}\|^{-2}$ the element $\lambda x^{\wedge} -x$ is Brown-Pedersen quasi-invertible in $E$. Indeed, if we regard $\lambda x^{\wedge} -x$ as an element in $E_x\equiv C(\hbox{Sp}(x))$, the JB$^*$-subtriple of $E$ generated by $x$ (see page \pageref{equ subtriple single generated}), then $x-\lambda x^{\wedge}$ is invertible and positive in $E_x,$ and its range tripotent is $r(x-\lambda x^{\wedge}) = e\in \partial_{e} (E_1)$. By Hua's identity (cf. \eqref{Hua}), we have $$x-\lambda^{-1} x^{[3]}=\left(x^\wedge-(x-\lambda x^\wedge)^\wedge\right)^\wedge.$$

Given $0<\lambda <\min\{||x^{\wedge}||^{-2},||T(x)^{\wedge}||^{-2}\}$, since $T$ strongly preserves Brown-Pedersen quasi-invertible elements, and $x, \lambda x^{\wedge} -x, T(x)$, and $T(\lambda x^{\wedge} -x)$ are Brown-Pedersen quasi-invertible, we deduce that $$T(x)-\lambda ^{-1} T(x)^{[3]}=(T(x)^\wedge-(T(x) - \lambda
T(x)^\wedge)^\wedge)^\wedge $$ $$= (T(x^\wedge)-(T(x) - \lambda
T(x^\wedge))^\wedge)^\wedge= T(\left(x^\wedge-(x-\lambda x^\wedge)^\wedge\right)^\wedge) =T(x)-\lambda^{-1}T(x^{[3]}),$$ for every $0<\lambda$ as above, which proves the desired statement.
\end{proof}

The full meaning of Theorem \ref{non-empty-extr} (and the role played by \cite[Lemma 2.2]{JaSiTaPe15} in its proof) is more explicit in the following result,
whose proof follows the lines we gave in the just mentioned theorem but replacing Proposition \ref{cubes} with Proposition \ref{p cubes strongly BP preserver}.

\begin{theorem}\label{thm non-empty-extr strongly BP preserver triples} Let $E$ and $F$ be JB$^*$-triples with $\partial_{e}(E_1)\neq\emptyset$.
Suppose $T:E\to F$ is a linear map strongly preserving Brown-Pedersen quasi-invertible elements. Then $T$ is a triple homomorphism.$\hfill\Box$
\end{theorem}

We can state now our conclusions on linear maps strongly preserving Brown-Pedersen quasi-invertibility.

\begin{theorem}\label{thm them BP-primeCstar necessary condition unital Calgebras} Let $A$ and $B$ be unital $C^*$-algebras. Let $T:A\to B$ be a linear map strongly preserving Brown-Pedersen quasi-invertible elements. Then there exists a Jordan $^*$-homomorphism $S: A\to B$ satisfying $T(x) = T(1) S(x)$, for every $x\in A$.\smallskip

\noindent We further know that $$T(A) \subseteq T(1)T(1)^* B T(1)^*T(1), \ S (A) \subseteq T(1)^*T(1) B T(1)^* T(1),$$ and $S: A\to T(1)^* T(1) B T(1)^* T(1)$ is a unital Jordan $^*$-homomorphism.
\end{theorem}

\begin{proof} Since $T$ preserves extreme points, $v= T(1)\in \partial_{e} (B_1)$ is a partial isometry with \begin{equation}\label{eq 5.13} (1-vv^*) T(x) (1-v^*v)=0
\end{equation} for every $x\in A$. It follows from \eqref{partial} that  $v T(a)^* v =vv^*T(a) v^*v$, for every $a=a^*\in A$.\smallskip

Now, Theorem \ref{thm them BP-primeCstar necessary condition unital Calgebras} assures that $T$ is a triple homomorphism. Thus, we have \begin{equation}\label{eq 5.15} T(x) =T \{x,1,1\} = \{T(x),v,v\} = \frac12 (T(x) v^*v + vv^* T(x)),
\end{equation} and \begin{equation}\label{eq 5.16} T(x^*) =T \{1,x,1\} = \{v,T(x),v\} = vT(x)^*v,
\end{equation} for every $x\in A$. Identities \eqref{eq 5.13} and \eqref{eq 5.15} give: \begin{equation}\label{eq 5.17} T(x) =vv^* T(x)v^*v=vv^* T(x) = T(x)v^*v,
\end{equation} for every $x\in A$. Multiplying on the left by $v^*$ we get $$v^* T(x) = v^* T(x)v^*v = \hbox{(by \eqref{eq 5.16})}= T(x^*)^* v,$$ for every $x\in A$, which proves that $S= v^* T : A\to B$ is a symmetric operator. Furthermore, since $T$ is a triple homomorphism, we have $$S(x^2) = v^* T\{x,1,x\} = v^* \{T(x),v,T(x)\} = v^* T(x) v^* T(x) =S(x)^2,$$ for all $x\in A$, which guarantees that $S$ is a Jordan $^*$-homomorphism. The identity in \eqref{eq 5.17} gives $T(x) = vv^* T(x) = v S(x)$, for every $x\in A$. The rest is clear.
\end{proof}

\begin{remark}
Under the hypothesis of Theorem \ref{thm them BP-primeCstar necessary condition unital Calgebras} we can similarly prove that the mapping $S_1 : A\to B$, $S_1 (x) = T(x) T(1)^*$ is a Jordan $^*$-homomorphism and $T(x) = S_1 (x) v$, for every $x$ in $A$.
\end{remark}

If $v$ is an extreme point of the closed unit ball of a prime unital C$^*$-algebra $B$, then $1= vv^*$ or $v^*v=1$. Therefore, the next result is a straight consequence of the previous Theorem \ref{thm them BP-primeCstar necessary condition unital Calgebras}.

\begin{corollary}\label{cor thm BP-primeCstar necessary condition unital Calgebras} Let $A$ and $B$ be unital $C^*$-algebras with $B$ prime. Let $T:A\to B$ be a linear map strongly preserving Brown-Pedersen quasi-invertible elements. Then one of the following statements holds:\begin{enumerate}[$(a)$] \item $T(1)^* T(1)=1$, $T(1) T(1)^* T(a) = T(a)$, for every $a\in A$, and there exists a unital Jordan $^*$-homomorphism $S: A\to B$ satisfying $T(a) = T(1) S(a)$, for every $a\in A$;
\item $T(1) T(1)^*=1$, $T(a) T(1)^* T(1)= T(a)$, for every $a\in A$, and there exists a unital Jordan $^*$-homomorphism $S: A\to B$ satisfying $T(a) = S(a) T(1)$, for every $a\in A$.
\end{enumerate} $\hfill\Box$
\end{corollary}

\end{document}